\documentclass[11pt, reqno]{article}
\title{Optimal factor matchings for point processes on non-amenable unimodular graphs}
\author{Yinon Spinka \and Oren Yakir}
\date{January 2026}

\usepackage{amsmath,amsfonts,amssymb,amsthm,graphics,enumerate,mathtools,thmtools}
\usepackage{hyperref,dsfont}
\hypersetup{
    colorlinks=true,
    linkcolor=blue,
    citecolor=red,
    urlcolor=blue,
    pdfborder={0 0 0}
}

\usepackage{graphicx,xcolor,bbm,fullpage}
\usepackage{geometry}
\newgeometry{vmargin={30 mm}, hmargin={25mm,25mm}}

\usepackage{cleveref}
  \crefname{theorem}{Theorem}{Theorems}
  \crefname{thm}{Theorem}{Theorems}
  \crefname{lemma}{Lemma}{Lemmas}
  \crefname{lem}{Lemma}{Lemmas}
  \crefname{remark}{Remark}{Remarks}
  \crefname{prop}{Proposition}{Propositions}
  \crefname{proposition}{Proposition}{Propositions}
  \crefname{notation}{Notation}{Notations}
  \crefname{claim}{Claim}{Claims}
  \crefname{observation}{Observation}{Observations}
  \crefname{defn}{Definition}{Definitions}
  \crefname{corollary}{Corollary}{Corollaries}
  \crefname{section}{Section}{Sections}
  \crefname{figure}{Figure}{Figures}
  \crefname{exercise}{Exercise}{Exercises}
    \crefname{assumption}{Assumption}{Assumptions}

\newtheorem{thm}{Theorem}[section]

\newtheorem{claim}[thm]{Claim}
\newtheorem{lemma}[thm]{Lemma}
\newtheorem{corollary}[thm]{Corollary}

\numberwithin{equation}{section}

\theoremstyle{definition}

\def\cU{\mathcal{U}}

\def\cP{\mathcal{P}}
\def\cO{\mathcal{O}}

\def\cI{\mathcal{I}}

\def\cG{\mathcal{G}}

\def\cC{\mathcal{C}}
\def\cB{\mathcal{B}}

\def\P{\mathbb{P}}
\def\E{\mathbb{E}}

\def\R{\mathbb{R}}
\def\Z{\mathbb{Z}}
\def\N{\mathbb{N}}

\def\R{\mathbb{R}}

\def\eps{\varepsilon}

\DeclareMathOperator{\dist}{dist}
\DeclareMathOperator{\diam}{diam}

\begin{document}
\maketitle

\begin{abstract}
Consider a unit-intensity point process $\Pi$ on the vertex set $V$ of a transitive non-amenable unimodular graph. We study invariant matchings between $\Pi$ and $V$ having small typical matching distances. When $\Pi$ is either a Poisson process or i.i.d.\ perturbations of the vertex set, we determine the optimal matching distance and show that it can be attained by a factor matching scheme (that is, a deterministic and equivariant function of $\Pi$).  
\end{abstract}

\section{Introduction}
\label{sec:intro}

Let $G = (V,E)$ be a locally finite transitive connected graph. A point process $\Pi$ on $G$ is a random multiset of vertices in $G$\footnote{All elements of the multiset are distinguishable. Formally, one may think of $\Pi$ as a set of the form $\{(v,i) : v \in V, 1 \le i \le \ell_v\}$, where $\ell_v \in \{0,1,\dots\}$ describes the number of occurrences of $v$.}. For $S\subset V$, we denote by $\Pi \cap S$ the points of $\Pi$ which are contained in $S$. We consider two types of point processes: the Poisson process and perturbed vertex sets. A \textbf{Poisson process} is a multiset $\Pi$ such that $\big\{|\Pi \cap \{v\}| \big\}_{v \in V}$ are independent Poisson(1) random variables.
A \textbf{perturbed vertex set} is a multiset of the form $\Pi=\{X_v : v \in V\}$, where $(X_v)_{v \in V}$ are independent random variables taking values in $V$, and $\gamma(X_v)$ has the same law as $X_{\gamma(v)}$ whenever $\gamma$ is an automorphism of $G$.
A natural example of a perturbed vertex set to have in mind is when each vertex $v\in V$ is moved to a uniformly random distance-$R_v$ neighbor, with $(R_v)$ i.i.d.\ $\N$-valued random variables. Both the Poisson process and the perturbed vertex set have laws which are invariant under the action of the automorphism group of $G$, and furthermore
$\E\big[|\Pi\cap \{v\}|\big] = 1$ for all $v\in V$.  
It is intuitively clear, and not hard to see, that the Poisson process is \emph{not} a perturbed vertex set.

In this work we study invariant matchings between $\Pi$ and $V$, or more generally, between $\Pi$ and $\Pi^\prime$, where $\Pi$ and $\Pi^\prime$ are independent and each is either a Poisson process or a perturbed vertex set. In our previous work~\cite{Elboim-Spinka-Yakir}, we dealt with analogous questions when $G= \Z^d$, see Section~\ref{subsec:related_works} below for more details. In the present work, we are interested in graphs with stronger expansion properties. Recall that $G=(V,E)$ is a \textbf{transitive} graph, meaning that for every two vertices $v_1,v_2\in V$ there exists an automorphism $\gamma$ of $G$ such that $\gamma(v_1) = v_2$. 
The \textbf{Cheeger constant} of $G$ is defined by
\begin{equation}
    \label{eq:cheeger_constant}
    h(G) = \inf\Big\{ \frac{|\partial A|}{|A|} \, : \, A\subset V \ \text{finite and non-empty} \Big\} \,  , 
\end{equation}
where $\partial A = \{ y\in V \, : \, d(y,A) = 1 \}$ is the external boundary of $A$.
The graph $G$ is called \textbf{amenable} if $h(G)=0$, and \textbf{non-amenable} if $h(G)>0$.
The \textbf{transition operator} is $\mathcal{P}:\ell^2(V) \to \ell^2(V)$ given by
\begin{equation*}
    (\mathcal{P}f)(v) = \frac{1}{d} \sum_{u\sim v} f(u) \, , 
\end{equation*}
As $\mathcal{P}$ is self-adjoint, it has real spectrum. The \textbf{spectral radius} $\rho$ of $\mathcal{P}$ is defined as
\[ \rho = \sup\{ |\lambda | \, : \, \lambda\not=1 \ \text{is an eigenvalue of } \mathcal{P} \} \, .\]
Kesten's criteria~\cite{Kesten} asserts that $G$ is non-amenable if and only if $\rho<1$. We say that a function $f:V\times V\to [0,\infty]$ is \textbf{diagonally invariant} if $f(v_1,v_2) = f(\gamma(v_1),\gamma(v_2))$ for any $v_1,v_2\in V$ and any automorphism $\gamma$. The graph $G$ is called \textbf{unimodular} if it satisfies the \textbf{mass transport principle}, that is for any diagonally invariant $f:V\times V\to [0,\infty]$ we have
\begin{equation}\label{eq:mass_transport_principle}\tag{MTP}
    \qquad \sum_{v\in V} f(v,x) = \sum_{v\in V} f(x,v) \qquad \text{for all } x\in V .
\end{equation}
In words, the total mass sent out of any given vertex $x$ is equal to the total mass received by $x$. Our main result will be stated for transitive non-amenable unimodular graphs. Some good examples to keep in mind for such graphs are the $d$-regular trees for $d\ge 3$, or any non-amenable Cayley graph with a finite set of generators. 

\subsection{Main result}

To state our main result we need to give some further definitions. A \textbf{matching} between two multisets $\Pi$ and $\Pi'$ in $V$ is an injective mapping $M \colon \Pi \to \Pi'$.
We call the matching \textbf{perfect} if $M$ is onto $\Pi'$, in which case, the inverse mapping $M^{-1} \colon \Pi' \to \Pi$ is a matching between $\Pi'$ and $\Pi$. An \textbf{invariant matching} between $\Pi$ and $\Pi'$ is a random matching $M : \Pi \to \Pi'$ such that the joint law of $(\Pi,\Pi',M)$ is invariant under the action of the automorphism group of $G$. We say that a perfect invariant matching $M:\Pi\to \Pi'$ is a \textbf{factor} (of $(\Pi,\Pi')$) if there exists a deterministic equivariant function $f$ such that $M = f(\Pi,\Pi')$ almost surely. By equivariant we mean that for any automorphism $T$ of $G$, we have that $f(T\circ(\Pi,\Pi'))=T\circ f(\Pi,\Pi')$ almost surely.

Our main result shows that there exists a factor perfect matching $M:\Pi\to \Pi^\prime$, with some properties. As a matter of fact, factor matchings need not exist for general graphs, as it may happen that the configuration seen from two different vertices is identical. We may avoid this technical obstruction in a simple and concrete manner by requiring that one can totally order the vertices in an equivariant manner. Specifically, we assume that
\begin{equation}\label{eq:assumption_total_order}
   \text{there exists an $\R$-valued factor}\footnote{Here, by factor we mean a factor of $(\Pi,\Pi')$. We note that if both point processes $\Pi,\Pi'$ are factors of a common Uni($[0,1]$)-valued i.i.d.\ process, and we are content with constructing our matching $M$ as a factor of this i.i.d.\ process, then the assumption becomes superfluous.}\text{ $(\cO_v)_{v \in V}$ so that a.s.\ $\{\cO_v\}$ are distinct. }  \tag{TO} 
\end{equation}
In Section~\ref{sec:sufficient_conditions_for_total_order} we give some sufficient conditions under which~\eqref{eq:assumption_total_order} holds. These conditions are easy to verify for $d$-regular trees and many other natural examples.

\begin{thm}
    \label{thm:factor_matching_for_transitive_non-amenable_graphs}
    Let $G$ be a connected transitive non-amenable unimodular graph and let $b_r$ denote the volume of a ball of radius $r$ in $G$. Let $\Pi$ be a perturbed vertex set or a Poisson process on $G$. Let $\Pi'$ be another such process (of either type), independent of $\Pi$.
    Assume that~\eqref{eq:assumption_total_order} holds.
    Then there exists a constant $c>0$ and a factor perfect matching $M$ between $\Pi$ and $\Pi'$ such that
    \[
    \E \big| \big\{ x \in \Pi \cap \{v\} : \dist(M(x),x) \ge r\}\big| \le \exp(-cb_r) 
    \]
    for all $v\in V$ and for all $r$ large enough.
\end{thm}

\noindent
Curiously, our proof of \cref{thm:factor_matching_for_transitive_non-amenable_graphs} does not use the independence between $\Pi$ and $\Pi'$ in any meaningful way, and the result holds as long as the two processes are factors of a common i.i.d.\ process.
Let us also mention that the constant $c$ in the theorem does not depend on the point processes $\Pi$ and $\Pi'$, and in fact, depends on the graph $G$ only through its degree, its Cheegar constant, and its spectral radius.

\subsection{Optimal tail for matching distance}
To better explain the content of Theorem~\ref{thm:factor_matching_for_transitive_non-amenable_graphs}, we suppose for a moment that $\Pi = V$ is simply the vertex set itself (note that $V$ can be thought of as a perturbed vertex set, where the perturbation is degenerate). 
Given an invariant matching $M$ between $V$ and $\Pi'$, we consider the matching distance $\dist(M(v),v)$, whose distribution does not depend on $v$. 
We shall be interested in its tail behavior, namely, 
\begin{equation} \label{eq:intro_tail_for_M_v}
    \P \big(\dist(M(v),v) \ge r\big) 
\end{equation}
as $r\to \infty$. 
What is the optimal tail behavior for the matching distance?
To gain some insight into this, consider the ``hole probability'' for $\Pi'$, given by
\[ h(r) := \P\big(|\Pi'\cap B_r(v)| = 0\big) .\]
This provides a lower bound on the matching distance under any invariant matching $M$ between $V$ and $\Pi^\prime$. Indeed, when there are no points of $\Pi'$ in a ball of radius $r$ around $v$, it must be the case that $v$ is matched to a point at distance greater than $r$. We thus have, 
$$\P\big(\dist(M(v),v)>r\big) \ge h(r).$$
For the Poisson process we obviously have $h(r) = \exp(-b_r)$, where $b_r=|B_r(v)|$. We also note that for many perturbed vertex sets we have 
\[ h(r) = e^{-\Theta(b_r)} \qquad\text{as } \ r \to \infty .\]
Indeed, the upper bound always holds, but $h(r)$ could be smaller, e.g., if there is no perturbation at all, or if the perturbation is of bounded distance.
Theorem~\ref{thm:factor_matching_for_transitive_non-amenable_graphs} shows that this tail behavior is achievable by an invariant perfect matching, which is furthermore a factor of the point process $\Pi'$. We note that every perturbed vertex set admits a canonical matching, obtained by moving each perturbed point $X_v$ back to its original position $v$. In general, the matching distance tail for this canonical matching does \emph{not} achieve the hole probability, and the matching given by Theorem~\ref{thm:factor_matching_for_transitive_non-amenable_graphs} is in fact better behaved. 

In the general situation, when $\Pi$ is also random, given an invariant matching $M$ between $\Pi$ and $\Pi'$, there may be many points (or no point) of $\Pi$ at a given $v \in V$, and it no longer makes sense to consider the matching distance $\dist(M(v),v)$. In this case, the quantity of interest is
\[ \E |\{ x \in \Pi \cap \{v\} : \dist(M(x),x) \ge r\}| ,\]
which again does not depend on $v$. Note that this is equal to~\eqref{eq:intro_tail_for_M_v} in the case $\Pi=V$.

\subsection{Related works}
\label{subsec:related_works}

While we do not dive into the details on how we construct the matching just yet, we remark that a key construction in our proof of Theorem~\ref{thm:factor_matching_for_transitive_non-amenable_graphs} is largely inspired by the work of Lyons and Nazarov~\cite{lyons2011perfect}. 
In that paper, they prove that for any bipartite non-amenable Cayley graph there exists an invariant perfect matching which is a factor of i.i.d.\ uniform random variables on $[0,1]$. In our proof, we construct a (random) bipartite graph between $\Pi$ and $\Pi^\prime$, and show that, after suitable modifications, the Lyons--Nazarov algorithm for constructing a matching yields the desired matching for us as well; see Section~\ref{sec:lyons-nazarov_matching_algo} below for a concrete description of the algorithm. 
We note that a big chunk of our work goes into establishing desirable properties of this random bipartite graph, whereas in~\cite{lyons2011perfect} similar properties where automatic as the bipartite graph was apriori given.
We also mention that Cs\'oka and Lippner~\cite{Csoka-Lippner} extended the Lyons--Nazarov result to non-amenable Cayley graphs which are not necessarily bipartite. Related problems in the context of Borel graphs and graphings have also been studied (see, e.g., \cite{bowen2021perfect,kastner2023baire,kun2021measurable,marks2016baire} and references therein).

A motivation for studying optimal matchings between point processes on non-amenable graphs comes from our recent work~\cite{Elboim-Spinka-Yakir}, where we addressed a related problem on the Euclidean lattice~$\Z^d$. In~\cite{Elboim-Spinka-Yakir}, we proved that under mild assumptions on the perturbation, one can construct an invariant perfect matching between the perturbed vertex set and the lattice points in $\Z^d$ with optimal tail bounds on the matching distance. In contrast, for the Poisson point process the analogous question is well known to be delicate and strongly dependent on the lattice dimension $d \ge 1$; see~\cite{Hoffman-Holroyd-Peres, Holroyd-Pemantle-Peres-Schramm, timar2023factor}. Moreover, it is still unknown whether a \emph{factor} matching exists between $\Z^d$ and its random perturbations, see discussions in~\cite[Section~4.1]{Elboim-Spinka-Yakir}.

Our Theorem~\ref{thm:factor_matching_for_transitive_non-amenable_graphs} shows that, in the non-amenable setting, one can indeed construct an invariant matching that both achieves optimal tails and is a factor of the underlying point processes. In this setup there is no behavioral distinction between the Poisson process and perturbed vertex sets, contrary to the Euclidean case. We mention that a discussion in Lyons~\cite{lyons-CPC} suggests that factor problems can sometimes become harder in non-amenable settings. 

We conclude the introduction by mentioning the recent surge of interest in studying point processes on trees and other hyperbolic spaces, as considered in this work. We refer the interested reader to~\cite{Benjamini-Krauz-Paquette, Bjorklund-Bylehn, Bylehn} and references therein.

\subsubsection*{Acknowledgments}
We thank Dor Elboim for very helpful discussions.
The research of Y.S. is supported in part by ISF grant 1361/22. The research of O.Y. is supported in part by NSF grant DMS-2401136. 

\section{A random bipartite graph with good expansion}
\label{sec:preparations}

Recall that $G$ is a connected, transitive, non-amenable, unimodular graph on vertex set $V$. In particular, it satisfies the mass transport principle~\eqref{eq:mass_transport_principle}.
Throughout, $\Pi$ is either the Poisson process on $G$ or a randomly perturbed vertex set, as described in the Section~\ref{sec:intro}, and $\Pi'$ is another such process (of either type). We do not require that $\Pi$ and $\Pi'$ are independent, but rather only assume that both $\Pi$ and $\Pi'$ are factors of some common i.i.d.\ process $\cI$ on $V$ (which is clearly the case when the two processes are independent).

\subsection{The bipartite graph}
\label{subsec:bipartite_graph}

We shall define a random bipartite directed graph $\cG$ on $\Pi \sqcup \Pi'$.
The directed edges will not play an essential role (later we simply forget the directions and consider the underlying undirected graph), but arise naturally in the construction, and we keep them for now as they also serve an instructional purpose. This graph will be defined via two $\N$-valued processes $\{R_v\}_{v\in V}$ and $\{R'_v\}_{v\in V}$, which can be thought of as the invariant maximal lengths for the matching $M$ that we eventually construct. Given these processes, we define $\cG$ as follows: 
\begin{itemize}
    \item there is a directed edge from $x\in \Pi$ to $x'\in \Pi'$ whenever $\text{dist}(x,x') \le R_x$; and 
    \item there is a directed edge from $x'$ to $x$ whenever $\text{dist}(x,x') \le R'_{x'}$.
\end{itemize}
Here and throughout, we slightly abuse the notation and view $x\in \Pi$ (or in $\Pi^\prime$) also as a vertex in $V$, and work with the graph distance.
Our matching $M$ will be a matching of the undirected graph underlying $\cG$, i.e., each $\{x,M(x)\}$ will be an edge of this graph. In particular, 
\begin{equation}
\label{eq:bound_on_matching_distance_by_r}
\dist(x,M(x)) \le \max\big\{R_x,R'_{M(x)}\big\} \qquad \text{for} \ x\in \Pi \, .    
\end{equation}
All definitions will be deterministic and equivariant functions of $(\Pi,\Pi')$, so that the obtained $M$ will indeed be a factor matching. The random variables $R_v$ and $R'_v$ will have the correct tails, so that the matching distance will satisfy the desired bound.

\noindent
We start by defining some parameters that will be used in the construction. Recall that $B_r(v)$ denotes the ball of radius $r$ around the point $v\in V$. We also write $b_r = |B_r(v)|$ for the size of the ball, and $d$ for the degree of $v$ (both, by transitivity, do not depend on $v\in V$). Let $h>0$ be the Cheeger constant of $G$, and let $\rho \in (0,1)$ be the spectral radius of $G$. Let $r_0$ be a sufficiently large even integer, depending only on the degree, Cheegar constant and spectral radius of $G$, and set
\begin{equation}
    \label{eq:def_of_bad_points}
    \cB := \{ v\in V \, : \, |\Pi' \cap B_{r_0/2}(v)| \le 0.9 \,  b_{r_0/2} \} \, . 
\end{equation}
For $v\in V$, we denote by $\cC_r(v)$ the collection of all finite $r$-connected sets\footnote{A subset $S \subset V$ is called $r$-connected if any two $u,v \in S$ are linked by a sequence $u=w_0,\dots,w_n=v$ such that $w_i \in S$ and $\dist(w_i,w_{i-1}) \le r$ for all $1 \le i \le n$. } containing $v$. We write $U^{+r}$ for the ball of radius $r$ around a set $U \subset V$. We now define $\{R_v\}_{v\in V}$ as follows: 
\begin{align}\label{eq:def_of_R_v_in_nonamenable_case} \nonumber
    R_v := r_0 \quad  &\text{whenever \  $\cB \cap B_{r_0/2}(v) = \emptyset$ \  and \  $|\Pi \cap \{v\}| \le r_0$};  \quad \text{and otherwise} \\
R_v &:= \min\Big\{ r>r_0 \, : \, |\Pi'\cap U^{+r}| \ge r |\Pi \cap U| \ \text{for every } U\in \cC_{4r}(v) \Big\} \, .
\end{align} 
\noindent We define $\{R'_v\}_{v \in V}$ analogously, by exchanging the roles of $\Pi$ and $\Pi'$.
This completes the definition of the bipartite graph $\cG$.

We now state two technical lemmas (we state these for $R_v$, but by symmetry they hold also for $R'_v$).
The first lemma will imply that our matching has the correct distance tails, as seen from~\eqref{eq:bound_on_matching_distance_by_r}.
The second lemma will allow us to use expansion properties of finite subsets in $G$.

\begin{lemma}
    \label{lemma:tail_bound_for_non-amenable_R}
     For any $r>r_0$, we have $\P(R_v > r) \le \exp(-cb_r)$.
\end{lemma}

\begin{lemma}
    \label{lemma:almost_sure_no_infinite_connected_comp_for_R_large}
     For any $r \ge r_0$, the set $\{v\in V\, : \, R_v > r\}$ almost surely contains no infinite $4r$-connected component.
\end{lemma}

We postpone the proofs of these lemmas to Section~\ref{sec:proof_of_lemmas}.

\subsection{The expansion}

Recall that $\Pi$ and $\Pi'$ are factors of an i.i.d.\ process $\cI$.
Let $A\colon V\to \Z_{\ge 0}$ be a factor of $\cI$. When $A$ takes only the values $\{0,1\}$, we will slightly abuse the notation and think of $A$ as a random subset of $V$. For general $A$, we think of $A$ as describing the number of points in a random multiset of $V$. We denote $p(A) = \E[A_v]$, which does not depend on $v\in V$. 
Note that $p(V)=p(|\Pi|)=p(|\Pi'|)=1$, where $|\Pi|_v := |\Pi \cap \{v\}|$ for $v\in V$. To lighten on the notation, when $\Lambda$ is a point process on $G$ (e.g.\ $\Pi$, $\Pi^\prime$ or a subset of these) we will slightly abuse the notation and write $p(\Lambda)$ for $p(|\Lambda|)$. 

Recall the definition of the bipartite directed graph $\cG$ on $\Pi\sqcup\Pi'$. For $A\subset \Pi\sqcup\Pi'$, we denote by $N_{\cG}(A) \subset \Pi\sqcup\Pi'$ the neighbors of vertices from $A$ in $\cG$. In particular, $N_{\cG}(A) \subset \Pi$ when $A \subset \Pi'$, and $N_{\cG}(A) \subset \Pi'$ when $A \subset \Pi$. The next lemma provides the engine for the proof of \cref{thm:factor_matching_for_transitive_non-amenable_graphs}, in which we shall construct a perfect matching in $\cG$. 
\begin{lemma}
    \label{lemma:boosted_hall}
    Let $A\subset \Pi$ or $A \subset \Pi'$ be a factor of the i.i.d.\ process $\cI$. Then
    \[
    p\big(N_\cG(A)\big) \ge \min\big\{2 \cdot p(A),\, \tfrac45 \big\} \, .
    \]
\end{lemma}
\noindent
The constants $2$ and $\frac45$ have no particular meaning, and by suitably modifying the constants in our definitions, the first constant could be made arbitrarily large and the second could be made arbitrarily close to 1.

A key step in the proof of Lemma~\ref{lemma:boosted_hall} is the following observation: When $A\subset V$ is a factor of i.i.d., the density of $N(A)$ (the neighbors of $A$ in $G$) must grow by at least a positive factor relative to the density of $A$, even when $A$ has infinite connected components. This is formalized in the next lemma. Recall that $G$ is connected, transitive, non-amenable, unimodular graph, and $\rho$ is its spectral radius.

\begin{lemma}
    \label{lemma:lyons_nazarov_style_chebyshev}
    Let $A\subset V$ be a factor of i.i.d., and denote $p=p(A)$ and $p^\prime = p(N(A))$. Then
    \[
    p^\prime \ge \frac{p}{\rho^2(1-p) + p}.
    \] 
\end{lemma}
\noindent Lemma~\ref{lemma:lyons_nazarov_style_chebyshev} originally appeared in~\cite[Lemma~2.3]{lyons2011perfect}, in the case where $G$ is a non-amenable Cayley graph. Their proof works just as well for transitive unimodular graphs, and for completeness we provide it below.
\begin{proof}[Proof of Lemma~\ref{lemma:lyons_nazarov_style_chebyshev}]
    We first note that
    \begin{equation*}
        p = \E\big[\mathbf{1}_{\{v\in A\}}\big] = \E\Big[\frac{1}{d} \sum_{u\sim v} \mathbf{1}_{\{u\in A\}}\Big] = \E\Big[\frac{1}{d} \sum_{u\sim v} \mathbf{1}_{\{u\in A , v \in N(A) \}} \Big] \, . 
    \end{equation*}
    Therefore, by the Cauchy-Schwarz inequality, 
    \begin{equation*}
        p^2 \le \E\Big[ \mathbf{1}_{\{v\in N(A)\}}\Big] \cdot \E\Big[\Big(\frac{1}{d} \sum_{u\sim v} \mathbf{1}_{\{u\in A\}}\Big)^2\Big] = p^\prime \cdot \bigg(\text{Var}\Big(\frac{1}{d} \sum_{u\sim v} \mathbf{1}_{\{u\in A\}}\Big) + p^2\bigg) \, .
    \end{equation*}
    By the definition of the spectral radius we have
    \[
    \text{Var}\Big(\frac{1}{d} \sum_{u\sim v} \mathbf{1}_{\{u\in A\}}\Big) \le \rho^2 \text{Var}\big( \mathbf{1}_{\{v\in A\}}\big) = \rho^2 p (1-p) \, ,
    \]
    and so we get that
    \[
    p \le p^\prime\big(\rho^2(1-p) + p\big),
    \]
    which is what we wanted to show. 
\end{proof}

\begin{proof}[Proof of Lemma~\ref{lemma:boosted_hall}]
Suppose without loss of generality that $A\subset \Pi$.
Denote
\[ A_0 := \{ x \in A : R_x=r_0 \} .\]
The proof splits into two cases, according to whether $A$ is almost entirely made up of $A_0$ or not. 
Assume first that $p(A_0) \ge 0.9 \cdot p(A)$. 
By the definition of $\cG$, we have
\[
p(N_{\cG}(A)) \ge p(N_\cG(A_0)) = p(\Pi'\cap V(A_0)^{+r_0}) = p(\Pi'\cap D^{+r_0/2}),
\]
where $V(A_0)$ is the set of vertices that are represented in $A_0$ (i.e., $V(A_0)$ is $A_0$ without multiplicities), and we denote $D:=V(A_0)^{+r_0/2}$. By the definition of $R_v$, we know that $D\cap \cB = \emptyset$. By the definition of $\cB$ (see~\eqref{eq:def_of_bad_points}), this implies that for each $v\in D$ we have $|\Pi'\cap B_{r_0/2}(v)| \ge 0.9 b_{r_0/2}$. By the mass transport principle we see that
\begin{align*}
    0.9 b_{r_0/2} \cdot p(D) &= 0.9 b_{r_0/2} \cdot  \E[\mathbf{1}_{\{v\in D\}}] \\ & \le \E\bigg[ \sum_{u\in V} \mathbf{1}_{\{v\in D\}} \mathbf{1}_{\{\text{dist}(u,v) \le r_0/2\}} \big|\Pi'\cap \{u \} \big|\bigg] \\  \eqref{eq:mass_transport_principle} &= \E\bigg[ \sum_{u\in V} \mathbf{1}_{\{u\in D\}} \mathbf{1}_{\{\text{dist}(u,v) \le r_0/2\}} \big|\Pi'\cap \{v \} \big|\bigg] \le b_{r_0/2} \cdot p(\Pi'\cap D^{+r_0/2}) \, ,
\end{align*}
where in the last inequality we used the fact that each point from $\Pi'\cap D^{+r_0/2}$ is counted at most $b_{r_0/2}$ times in the sum. Altogether, we arrive at the inequality
\begin{equation}
    \label{eq:lower_bound_of_p_N_G_A_first_case}
    p(N_{\cG}(A)) \ge 0.9 \cdot p(D)  = 0.9 \cdot p(V(A_0)^{+r_0/2})\, .
\end{equation}
By the definition of $R_v$, we have that $|\Pi \cap \{v\}| \le r_0$ for all $v \in A_0$.
Thus, $p(V(A_0)) \ge \frac1{r_0}p(A_0)$.
Now, if $p(V(A_0)^{+r_0/2}) \ge 0.9$,  then~\eqref{eq:lower_bound_of_p_N_G_A_first_case} implies that 
\[
p(N_{\mathcal{G}}(A)) \ge 0.9 \cdot p(V(A_0)^{+r_0/2}) \ge (0.9)^2  \ge \tfrac45 \, . 
\]
On the other hand, if $p(V(A_0)^{+r_0/2}) <0.9$, we can repeatedly apply Lemma~\ref{lemma:lyons_nazarov_style_chebyshev} (as $(\Pi,\Pi')$ is a factor of the i.i.d.\ process $\cI$) and obtain using~\eqref{eq:lower_bound_of_p_N_G_A_first_case} that
\begin{align*}
    p(N_\cG(A))
    &\ge 0.9 \cdot p(V(A_0)^{+r_0/2}) \\
    &\ge 0.9 \cdot(0.1\rho^2 + 0.9)^{-r_0/2} \cdot p(V(A_0)) \\
    &\ge 0.9 \cdot(0.1\rho^2 + 0.9)^{-r_0/2} \cdot \tfrac1{r_0} \cdot p(A_0) \\
    &\ge (0.9)^2 \cdot (0.1\rho^2 + 0.9)^{-r_0/2} \cdot \tfrac1{r_0} \cdot p(A) \ge 2 \cdot p(A)\, ,
\end{align*}
where the last inequality holds for $r_0$ sufficiently large.

Assume now that $p(A_0) < 0.9 \cdot p(A)$.
For $k \ge 1$, denote
\[
A_k := \{ x \in A : 2^{k-1} r_0 < R_x \le 2^k r_0\} \, ,
\]
and observe that
\[ p(A) \le 10 \sum_{k=1}^\infty p(A_k) \le 10 \cdot \sup_{k \ge 1} \big[2^k \, p(A_k)\big]\sum_{k=1}^\infty 2^{-k} = 10 \cdot \sup_{k \ge 1} \big[2^k \, p(A_k)\big] \, . \]
Therefore, it suffices to show that for every $k \ge 1$,
\begin{equation}
    \label{eq:lower_bound_of_p_N_G_A_second_case}
 20 \cdot 2^k \, p(A_k) \le p(N_{\cG}(A)) .
\end{equation}
Since $V(A_k) \subset \{ v \in V : R_v > 2^{k-1}r_0 \}$, \cref{lemma:almost_sure_no_infinite_connected_comp_for_R_large} implies that $V(A_k)$ consists only of finite $2^{k+1} r_0$-connected components, almost surely.
Thus, if $U$ is one such component, then choosing any $v \in U$, noting that $r := R_v \in [2^{k-1}r_0,2^kr_0]$ so that $U$ is $4r$-connected, and using the definition~\eqref{eq:def_of_R_v_in_nonamenable_case} of $R_v$, we see that $|\Pi \cap U| \le \tfrac1r |\Pi' \cap U^{+r}|$. By choosing $v \in U$ which minimizes $R_v$, $v \in U$, we get that $\Pi' \cap U^{+r} \subset N_\cG(U) \subset N_\cG(A_k)$.
Therefore, using that $\{U^{+2^kr_0}\}$ are pairwise disjoint as $U$ ranges over the $2^{k+1} r_0$-connected components of $V(A_k)$, an application of~\eqref{eq:mass_transport_principle} (where each point in $A_k$ sends out a unit mass to each of its neighbors in $N_{\cG}(A)$), yields that
\[
p\big(A_k\big) \le \frac1{2^{k-1}r_0} \cdot p\big(N_{\cG}(A_k)\big) \le \tfrac1{20} 2^{-k} \cdot p\big(N_{\cG}(A)\big) .
\]
This completes the proof of the lemma.
\end{proof}
\noindent
We end this section with a simple claim about independent sets in $\cG$.
\begin{claim}\label{cl:indep-set-density}
    Let $A\subset \Pi \sqcup \Pi'$ be a factor of $\mathcal{I}$ such that $A$ is almost surely an independent set in $\cG$. Then
    \[
    \min\{ p(A \cap \Pi),\, p(A \cap \Pi') \} \le \tfrac13 .
    \]
\end{claim}
\begin{proof}
    Since $N_\cG(A \cap \Pi) \cap (A \cap \Pi') = \emptyset$, we have that
    \[ p(N_\cG(A \cap \Pi)) + p(A \cap \Pi') \le 1 .\]
    If $p(A \cap \Pi) \le \tfrac13$, we are done. Otherwise, by \cref{lemma:boosted_hall},
    \[ p(A \cap \Pi') \le 1 - p(N_\cG(A \cap \Pi)) \le 1 - \min\{2p(A \cap \Pi),\tfrac45\} \le \tfrac13 . \qedhere \]
\end{proof}

\section{The Lyons-Nazarov matching algorithm}
\label{sec:lyons-nazarov_matching_algo}

In this section we describe the matching algorithm which yields the proof of Theorem~\ref{thm:factor_matching_for_transitive_non-amenable_graphs}. As we already mentioned in the introduction, the idea is inspired by a similar matching algorithm from~\cite{lyons2011perfect}, though the details of the proof are somewhat different. 
Recall the definition of the bipartite directed graph $\cG$ on $\Pi \sqcup \Pi'$ from Section~\ref{subsec:bipartite_graph}. In this section, we forget about the direction of the edges, and only consider $\cG$ as an undirected graph, in which case, $x\in \Pi$ and $x'\in \Pi'$ share an edge whenever $\text{dist}(x,x') \le \max\{R_x,R'_{x'}\}$. Our goal is to construct a factor matching scheme $M:\Pi\to \Pi'$ so that $\text{dist}(x,M(x)) \le \max\{R_x,R'_{M(x)}\}$ almost surely.

A \textbf{partial matching} in $\cG$ is a collection of disjoint edges of $\cG$. Suppose we are given a partial matching. A path in $\cG$ is said to be \textbf{alternating} if its edges alternate between belonging to the matching and not. A \textbf{chain} is a simple alternating path that starts and ends at unmatched vertices. By the \textbf{flipping} of a chain $(x_1,y_1,\dots,x_n,y_n)$, we mean the removal of the existing matching edges $\{ \{y_i,x_{i+1}\} : 1 \le i \le n-1 \}$ and insertion of new matching edges $\{ \{x_i,y_i\} : 1 \le i \le n \}$, thereby producing a new partial matching, in which all $x_1,\dots,x_n$ and $y_1,\dots,y_n$ are now matched. See Figure~\ref{figure:flipping} for a visual demonstration.  The \textbf{length} of a chain is the number of edges it contains (e.g., $2n-1$ for the chain just mentioned). A chain necessary has odd length.

   	\begin{figure}
		 	\begin{center}	\scalebox{0.3}{\includegraphics{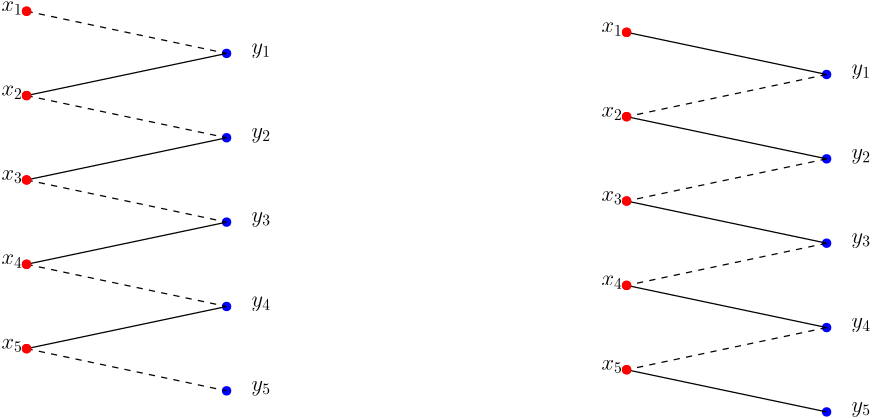}}
		 	\end{center}
            \caption{Left: A chain of length 9, beginning at the red vertex $x_1$ and ending at the blue vertex $y_5$. Dashed edges correspond to edges not contained in the partial matching. Right: The corresponding chain after flipping. Solid edges correspond to edges contained in the (new) partial matching.}
            \label{figure:flipping}
		\end{figure}

\subsection{Proof of main result}

In this section we provide the proof of Theorem~\ref{thm:factor_matching_for_transitive_non-amenable_graphs}. Indeed, we will construct the desired matching in a sequence of infinitely many stages, where at each stage we only have a partial matching. The set of matched vertices will only increase at each stage (but the set of matched edges will not increase), and each fixed vertex will be matched at some finite stage. We shall do this in such a manner that each edge changes its state only finitely often almost surely. This will give the desired factor matching by taking the pointwise limit which exists almost surely.

\begin{proof}[Proof of Theorem~\ref{thm:factor_matching_for_transitive_non-amenable_graphs}]

    We start with the empty partial matching. At the end of the $n$-th stage there will be no chains of length less than $4n$. Stage $n$ proceeds as follows: Consider all chains of length less than $4n$ (in this paragraph, a chain always refers to these). We place a total order on this collection, by taking the lexicographical order induced by a total order on $V$, which in turn exists by the assumption~\eqref{eq:assumption_total_order}.
    Now consider the subset of chains which are the smallest among all chains intersecting it, and flip all of those chains simultaneously. After that, reevaluate the set of chains of length less than $4n$ and repeat the previous step. After countably many repetitions of this step, there are no more chains of length less than $4n$, and this is the end of stage $n$.

    It remains to show that each edge of $\cG$ changes its state only finitely many times almost surely. Define
    \begin{align*}
     p_n &:= p(\{ x \in \Pi : x\text{ is unmatched at the end of stage }n\}) \\
         &\phantom{:}= p(\{ x' \in \Pi' : x'\text{ is unmatched at the end of stage }n\}) ,
    \end{align*}
    with the equality following from a simple application of~\eqref{eq:mass_transport_principle}.
    It suffices to show that $p_n$ decays exponentially in $n$.
    Indeed, given this, a mass transport argument then gives that the expected number of times that an edge changes its state is finite (if each endpoint of a flipped chain sends mass 1 to each vertex along the chain, then the expected mass out of a given vertex is at most $\sum_n 4n p_n < \infty$, and the mass in is the number of times that an incident edge flips its state).

    Let us show that $p_n$ decays exponentially.
    Fix $n$ and consider the partial matching at the end of stage $n$.
    Let $A_0$ denote the set of unmatched vertices (belonging to either $\Pi$ or $\Pi'$), and define $B_0 := N_\cG(A_0)$.
    For $k \ge 1$, let $A_k$ be all vertices that are matched with someone from $B_{k-1}$, and define $B_k := N_\cG(A_k)$.
    Observe that $A_k$ (resp.\ $B_k$) is the set of vertices $x$ for which there exists an alternating path of even (resp.\ odd) length at most $2k+1$ from an unmatched vertex to $x$.
    Since there are no chains of length less than $4n$, every vertex in $B_0 \cup \cdots \cup B_{n-1}$ is matched, and each of $A_0,A_1,\dots,A_{n-1}$ is an independent set in $\cG$.
    It follows from the former that $A_1 \subset A_2 \subset \cdots \subset A_{n-1}$, and together with~\eqref{eq:mass_transport_principle} that $p(A_k)=p(B_{k-1})$, $p(A_k \cap \Pi)=p(B_{k-1} \cap \Pi')$ and $p(A_k \cap \Pi')=p(B_{k-1} \cap \Pi)$ for $1 \le k \le n$.
    Since $A_{n-1}$ is an independent set in $\cG$, \cref{cl:indep-set-density} gives that
    \[ \min\big\{p(A_{n-1} \cap \Pi),\, p(A_{n-1} \cap \Pi')\big\} \le \tfrac13 .\]
    Suppose without loss of generality that $p(A_{n-1} \cap \Pi) \le \frac13$.
    Since $A_1 \subset A_2 \subset \cdots \subset A_{n-1}$, we have that $p(A_k \cap \Pi) \le \frac13$ for all $0 \le k \le n-1$ (for $k=0$ this follows from \cref{cl:indep-set-density} since $A_0$ is an independent set in $\cG$ and $p(A_0 \cap \Pi)=p(A_0 \cap \Pi')$).
    Thus, by \cref{lemma:boosted_hall},
    \[ p(A_{k+1} \cap \Pi) = p(B_k \cap \Pi') \ge 2 \cdot p(A_k \cap \Pi) \qquad\text{for all }0 \le k \le n-1 .\]
    Hence, $1 \ge p(A_n \cap \Pi) \ge 2^n p(A_0 \cap \Pi)$. Thus, $p_n=p(A_0 \cap \Pi) \le 2^{-n}$.

    \noindent
    This completes the proof that our construction stabilizes and yields a matching $M$ between $\Pi$ and $\Pi'$.
    This matching is a factor of $(\Pi,\Pi')$.
    Since $p_n \to 0$, all vertices are matched in $M$, meaning that $M$ is a perfect matching.
    Finally, for $u,v \in V$, the event that some $x \in \Pi \cap \{v\}$ is matched to some $y \in \Pi' \cap \{u\}$ is contained in the event that $\max\{R_v,R'_u\} \ge \dist(u,v)$, so that
    \begin{align*}
     \E |\{ x \in \Pi \cap \{v\} : \dist(M(v),v) \ge r\}|
      &\le \P\big(\dist(M(x),x) \ge r\text{ for some }x \in \Pi \cap \{v\}\big) \\
      &\le \sum_{u:\dist(u,v) \ge r} \P(\max\{R_v,R'_u\} \ge \dist(u,v)) \\
      &\le \sum_{k=r}^\infty b_k \cdot 2\P(R_v \ge k) \le e^{-c b_r} ,
    \end{align*}
    where we used \cref{lemma:tail_bound_for_non-amenable_R} in the last inequality.
\end{proof}

\section{Proofs of supporting lemmas}
\label{sec:proof_of_lemmas}

In this section we prove \cref{lemma:tail_bound_for_non-amenable_R} and \cref{lemma:almost_sure_no_infinite_connected_comp_for_R_large}.
For the reader's convenience, we recall some definitions.
Recall that the set $\cC_r(v)$ consists of all finite $r$-connected sets containing $v$, and
\begin{align*}
    R_v := r_0 \quad  &\text{whenever \  $\cB \cap B_{r_0/2}(v) = \emptyset$ \  and \  $|\Pi \cap \{v\}| \le r_0$};  \quad \text{and otherwise} \\
R_v &:= \min\Big\{ r>r_0 \, : \, |\Pi'\cap U^{+r}| \ge r |\Pi \cap U| \ \text{for every } U\in \cC_{4r}(v) \Big\} \, ,
\end{align*} 
where $\cB := \{ v\in V \, : \, |\Pi' \cap B_{r_0/2}(v)| \le 0.9 \,  b_{r_0/2} \}$.

\subsection{Tail bounds for the matching distance}

We first state a large deviations type bound for the discrepancy in the point count of $\Pi$ and $\Pi^\prime$.

\begin{claim}
    \label{claim:prob_bound_for_bad_point_count_in_U}
    For any finite set $U\subset V$ we have 
    \begin{equation*}
        %\label{eq:prob-bound-for-bad-point-count-in-U}
        \P\big(|\Pi'\cap U^{+r}| < r |\Pi \cap U| \big) \le e^{-c|U^{+r}|}\, .
    \end{equation*}
\end{claim}

We now provide the proof \cref{lemma:tail_bound_for_non-amenable_R} based on the above claim. 

\begin{proof}[Proof of \cref{lemma:tail_bound_for_non-amenable_R}]
    We need to show that $\P(R_v > r) \le \exp(-cb_r)$ for all $r>r_0$. Claim~\ref{claim:prob_bound_for_bad_point_count_in_U} together with a union bound, yields that
    %By the definition~\eqref{eq:def_of_R_v_in_nonamenable_case} of $R_v$, this is the same as
    \[ \P\big(R_v > r\big) = \P\big(|\Pi'\cap U^{+r}| < r |\Pi \cap U| \ \text{for some } U\in \cC_{4r}(v)\big) \le \sum_{U \in \cC_{4r}(v)} e^{-c|U^{+r}|} .\]
    The bound now follows by splitting the right-hand side according to the size of $|U|$. Indeed, as $\cC_{4r}(v)$ consists of $4r$-connected sets, we have 
     \begin{equation}
         \label{eq:exponential_bound_on_connected_sets_of_given_size}
         \# \{ U\in \cC_{4r}(v) \, : \, |U| = k \} \le e^{Crk},
     \end{equation} 
     for some constant $C>0$ which depends only on the maximum degree in $G$; 
     see, e.g.,~\cite[Chapter 45]{bollobas2006art}. Furthermore, by the definition of the Cheeger constant $h=h(G)>0$ (see~\eqref{eq:cheeger_constant}), we have that
    \[
    |U^{+r}| \ge (1+h)^r |U| .
    \]
    Since we also have that $|U^{+r}| \ge b_r$, we obtain that
    \begin{align*}
     \P(R_v > r) &\le \sum_{k\ge 1} e^{Crk} e^{-c\max\{b_r , (1+h)^r k \}} \\ &\le \sum_{k=1}^{\lfloor b_r/(1+h)^r\rfloor } e^{Crk}  e^{-cb_r} + \sum_{k\ge \lfloor b_r/(1+h)^r\rfloor + 1}e^{-k\big(c(1+h)^r-Cr
     \big)} \le e^{-cb_r} \, ,
    \end{align*}
    where the last inequality holds for $r_0$ sufficiently large and we are done.
\end{proof}

\begin{proof}[Proof of Claim~\ref{claim:prob_bound_for_bad_point_count_in_U}]
    
    By separating into ``undercrowding'' and ``overcrowding'' events, it suffices to show that
    \begin{equation}\label{eq:prob-bound-too-few-points}
    \P(|\Pi'\cap U^{+r}| \le 0.9 |U^{+r}|) \le e^{-c|U^{+r}|} .
    \end{equation}
    and
    \begin{equation}\label{eq:prob-bound-too-many-points}
    \P(|\Pi \cap U| \ge \tfrac{0.9}{r} |U^{+r}|) \le e^{-c|U^{+r}|} .
    \end{equation}
    We begin with the undercrowding event~\eqref{eq:prob-bound-too-few-points}.
    For this, we show more generally that for any finite set $W\subset V$ and any $\eps>0$ we have
    \begin{equation}
        \label{eq:concentration_undercrowding}
        \P\Big(|\Pi'\cap W| \le (1-\eps) |W|\Big) \le \exp \big(-\tfrac12 \eps^2 |W|\big) \, .
    \end{equation}
    We first prove~\eqref{eq:concentration_undercrowding} in the case where $\Pi'$ is a perturbed vertex set. Indeed, by transitivity we have $\E|\Pi'\cap W| = |W|$. Furthermore, $|\Pi'\cap W| = \sum_{v\in V} \mathbf{1}_{\{\Pi'_v \in W\}}$ is a sum of independent Bernoulli random variables. Hence, the estimate~\eqref{eq:concentration_undercrowding} follows from the standard bound on the moment generating function of $|\Pi'\cap W|$, see e.g.~\cite[Theorem~A.1.13]{alon2016probabilistic}. In the case when $\Pi'$ is a Poisson process, this follows in the same way by approximating a Poisson random variable by sums of independent Bernoulli random variables.

    We now turn to overcrowding event~\eqref{eq:prob-bound-too-many-points}.
    For this, we show more generally that for any finite set $W\subset V$ and any $n \ge 0$ we have
    \begin{equation*}
        \P\Big(|\Pi\cap W| \ge |W| + n\Big) \le \exp\left(n-(|W|+n) \log(1+\tfrac{n}{|W|})\right) \, .
    \end{equation*}
    When $\Pi$ is a perturbed vertex set, this again follows from a standard bound, see e.g.~\cite[Corollary~A.1.10]{alon2016probabilistic}. When $\Pi$ is a Poisson process, this again follows by approximation as before.
    Using that $\log(1+x)-1 \ge \frac12 \log x$ for $x \ge 10$, the above implies that
    \begin{equation}
        \label{eq:concentration_overcrowding}
        \P\Big(|\Pi\cap W| \ge n\Big) \le \exp\left(-\tfrac12 n\log \tfrac{n}{2|W|}\right) \qquad\text{for }n \ge 10|W| .
    \end{equation}
    Finally, this implies~\eqref{eq:prob-bound-too-many-points} by taking $W=U$ and $n=\frac{0.9}{r}|U^{+r}| \ge 2(1+h)^{cr} |U|$. 
\end{proof}

\subsection{Absence of infinite clusters}

We now turn to \cref{lemma:almost_sure_no_infinite_connected_comp_for_R_large},for which we will need the following technical combinatorial claim. 
\begin{claim}
    \label{claim:maximal_shortest_path}
    Let $r\ge 1$. Let $\{U_i\}_{i\in I}$ be a family of $r$-connected sets in $G$ such that $U=\bigcup_{i\in I} U_i$ is also $r$-connected. Then for all $u,v\in U$ there exists ${j_1},\ldots,{j_n}\in I$ such that:
    \begin{enumerate}
        \item[{\normalfont (1)}] For all $k\not=k^\prime$ we have $\text{\normalfont dist}(U_{j_k},U_{j_{k^\prime}}) > r$;
        \item[{\normalfont (2)}] For all $1\le k < n$ we have $\text{\normalfont dist}(U_{j_k},U_{j_{k+1}}) \le  |U_{j_{k}}| + |U_{j_{k+1}}| + 3r $;
        \item[{\normalfont (3)}] We have $\text{\normalfont dist}(u,U_{j_{1}}) \le |U_{j_1}| + r$ and $\text{\normalfont dist}(v,U_{j_{n}}) \le |U_{j_n}| + r$. 
    \end{enumerate}
    In particular, we have $\text{\normalfont dist}(u,v) \le 3rn + 3\sum_{k=1}^{n} |U_{j_k}|$. 
    \end{claim}
\begin{proof}
    Consider the graph $\mathsf{G}$ with vertices given by the index set $I$, in which $i,i^\prime\in I$ are adjacent if and only if $\text{dist}(U_{i},U_{i^\prime}) \le r$. Note that $\mathsf{G}$ is connected, since $U$ is $r$-connected. Let $\{i_1,\ldots,i_m\}$ be a shortest path in $\mathsf{G}$ between $u\in U_{i_1}$ and $v\in U_{i_m}$. The desired index set $\{j_1,\ldots,j_n\}$ will be constructed as a sub-path of $\{i_1,\ldots,i_m\}$. We say that an index set $J\subset I$ is \emph{sparse} if it does not contain any two consecutive integers. Note that for any sparse $J\subset \{i_1,\ldots ,i_m\}$, the sets $\{U_{j}\}_{j\in J}$ are at pairwise distance greater than $r$, as otherwise we will have a contradiction to the fact that $\{i_1,\ldots,i_m\}$ is a shortest path. We construct $\{j_1,\ldots,j_n\}$ in a greedy manner as follows: At step, simply put $J_0 = \emptyset$. At step $k+1$, we construct $J_{k+1}$ by adding to $J_k$ the index $i\in \{i_1,\ldots,i_m\}\setminus J_k$ among which $|U_i|$ is largest and $J_k\cup\{i\}$ is sparse (if there are a few possible choices of indices, just take the smallest one). We stop when this is no longer possible, i.e., when the resulting set is a maximal sparse set. 

    Suppose now that the greedy process stopped after $n$ steps, i.e. that $J=J_n = \{j_1,\ldots,j_n\}$. We want to check that conditions (1), (2) and (3) hold for this index set. Condition (1) is immediate, since $J$ is sparse by construction. To verify condition (2), we first note that if $j_k = i_{\ell}$, then necessarily $j_{k+1} \in \{i_{\ell + 2},i_{\ell + 3} \}$. If $j_{k+1} = i_{\ell +2}$, then we must have
    \[
    |U_{i_{\ell + 1}}| \le \max \{U_{j_{k}},U_{j_{k+1}}\}
    \]
    by the greedy construction. Therefore
    \begin{align*}
        \text{dist}(U_{j_{k}},U_{j_{k+1}} ) &\le \text{dist}(U_{j_{k}},U_{i_{\ell + 1}} ) + |U_{i_{\ell + 1}}| + \text{dist}(U_{i_{\ell + 1}},U_{j_{k+1}} ) \\ & \le r + \max \{U_{j_{k}},U_{j_{k+1}}\} + r \\ &\le 2r + |U_{j_{k}}| + |U_{j_{k+1}}| \, .     
    \end{align*}
    In the case where $j_{\ell + 1} = i_{\ell +3}$, the greedy construction implies that $|U_{j_{k}}| \ge |U_{i_{\ell + 1}}|$ and $|{U_{j_{k+1}}}|\ge |U_{i_{\ell + 2}}|$, and we get that
    \[
    \text{dist}(U_{j_{k}},U_{j_{k+1}} ) \le |U_{i_{\ell+1}}|+ |U_{i_{\ell+2}}| + 3r \le  |U_{j_{k}}| + |U_{j_{k+1}}| + 3r \, .
    \]
    Altogether, condition~(2) holds. It remains to check condition~(3), which is pretty straightforward. If $j_1 = i_1$, then $\text{dist}(u,U_{j_1}) = 0$. Otherwise, the greedy construction implies that $j_1 = i_2$ and that $|U_{j_{1}}| > |U_{i_1}|$, and we get that
    \[
    \text{dist}(u,U_{j_1}) \le |U_{i_1}| + r \le |U_{j_1}| + r \, .  
    \]
    A similar argument shows that $\text{dist}(v,U_{j_n}) \le |U_{j_n}| + r$, and the claim follows. 
\end{proof}

\begin{proof}[Proof of Lemma~\ref{lemma:almost_sure_no_infinite_connected_comp_for_R_large}]
    We need to show that for any $r \ge r_0$, the set $\{v\in V\, : \, R_v > r\}$ almost surely consists only of finite $4r$-connected components.
    We first prove this in the case where $r>r_0$. Fix some $u\in V$ and let $\Gamma_r(u)$ be the $4r$-connected component of $\{v\in V \, : \, R_{v} > r\}$ which contains $u$. We want to show that $\diam(\Gamma_r(u))$ is finite almost surely. Indeed, by the definition~\eqref{eq:def_of_R_v_in_nonamenable_case} of $R_v$, if $\diam(\Gamma_r(u))\ge m$ then there exists a family $\{U_i\}_{i\in I}$ of $4r$-connected sets (one of which contains $u$) such that $U=\bigcup_{i\in I} U_i$ is $4r$-connected, $\diam(U)\ge m$, and such that
    \[
    |\Pi' \cap U_i^{+r} |  < r \, |\Pi \cap U_i| \qquad \text{for all } i\in I\, .
    \]
    By applying Claim~\ref{claim:maximal_shortest_path} to this family (with $4r$ in place of $r$), we conclude that there exists finite sub-family of $4r$-connected sets $\{U_{j}\}_{j=1}^{n}$ which are $4r$-separated and have
    \[
    \sum_{j=1}^{n} |U_j| \ge \frac{m}{30r} \, .
    \]
    Denote by $\cU = \bigcup_{j=1}^{n} U_j$, and note that since the $U_j$'s are $4r$-separated we have
    \begin{equation}        \label{eq:proof_of_lemma_almost_sure_no_infinite_connected_comp_for_R_large_deficiency_of_random_points}
        |\Pi' \cap  \cU^{+r} | < r \, |\Pi \cap \cU| \, .
        \end{equation}
    By \cref{claim:prob_bound_for_bad_point_count_in_U} we know that
    \[
    \P\big(|\Pi' \cap \cU^{+r} | < r \, |\Pi \cap \cU|\big) \le e^{-c |\cU^{+r}|} \, .
    \]
    Denote by $U_0 = \{u\}$, $s_j = |U_j|$ and $d_j = \dist(U_{j-1},U_j)$. Then by items (2) and (3) from Claim~\ref{claim:maximal_shortest_path} we know that $d_j \le s_{j-1} + s_j + 12r$ for all $1\le j \le n$. Furthermore, by denoting
    \[
    s := s_1+\ldots + s_n = \sum_{j=1}^{n}|U_j| = |\mathcal{U}| \, ,
    \]
    we observe that $|\mathcal{U}^{+r}| \ge (1+h)^{r} s$, where $h>0$ is the Cheeger constant of $G$. We will now union bound over all possible choices of $U_1,\ldots,U_n$ as above. Indeed, note that $1\le n\le s$, and that given $U_1,\ldots,U_{j-1}$ we can choose $U_j$ as follows: Pick a point at distance $d_j$ from $U_{j-1}$ (there are at most $s_{j-1} \cdot d^{d_j}$ such choices) and then choose $U_j$ as a $4r$-connected set of size $s_j$ (there are at most $e^{Crs_j}$ such choices, see~\eqref{eq:exponential_bound_on_connected_sets_of_given_size}). The union bound now gives
    \begin{equation}
        \label{eq:bound_on_tail_of_diameter_connected_R_componenet_after_union_bound}
        \P\big(\text{diam}(\Gamma_r(u))\ge m\big) \le \sum_{s = \lfloor m/9r \rfloor}^{\infty} \,  \sum_{n=1}^{s} \bigg(\sum_{s_1+\ldots+s_n = s} \,  \prod_{j=1}^n s_{j-1} d^{s_{j-1} + s_{j} + 12r} e^{Crs_j} \bigg) \, e^{-c(1+h)^{r} s}.
    \end{equation}
    Let us further bound the inner sum in~\eqref{eq:bound_on_tail_of_diameter_connected_R_componenet_after_union_bound}. Clearly, we have
    \[
    \prod_{j=1}^n s_{j-1} d^{s_{j-1} + s_{j} + 12r} e^{Crs_j} \le e^s \cdot d^{2s} \cdot d^{12rn} \cdot e^{Crs} \le e^{Crs} \, ,
    \]
    for some $C>0$ which depends only on the graph $G$.
    Furthermore, a crude bound on binomial coefficients gives
    \[
    \sum_{n=1}^{s}\Big(\sum_{s_1+\ldots+s_n = s} 1\Big)= \sum_{n=1}^{s} \binom{s+n-1}{n-1} \le e^{Cs} \, .
    \]
    Plugging these bounds into~\eqref{eq:bound_on_tail_of_diameter_connected_R_componenet_after_union_bound} yields that
    \[
    \P\big(\text{diam}(\Gamma_r(u))\ge m\big) \le \sum_{s = \lfloor m/9r \rfloor}^{\infty} \big(e^{Cr- c(1+h)^r}\big)^s  \le e^{-\frac cr (1+h)^r m},
    \]
    where the last inequality holds for $r_0$ sufficiently enough. This shows that $\text{diam}(\Gamma_r(u))$ is finite almost surely. Since the number of vertices is countable, we conclude that almost surely $\{v\in V \, : \, R_v>r \}$ does not have an infinite $4r$-connected component, for $r>r_0$.     
    
    It remains to deal with the case $r=r_0$. As in the previous case, we fix some $u\in V$ and denote by $\Gamma_{r_0}(u)$ the $4r_0$-connected component of $\{v\in V \, : \, R_{v}>r_0 \}$ which contains $u$.
    By definition, if $v\in V$ has $R_{v}>r_0$, this means that $B_{r_0}(v) \cap \cB \not=\emptyset$ or $|\Pi \cap \{v\}|>r_0$, where $\cB$ was defined in~\eqref{eq:def_of_bad_points}. Therefore, if $\diam(\Gamma_{r_0}(u)) \ge m$ for $m$ large enough, then we can find $v_1,\ldots,v_n$ such that 
    \begin{enumerate}
        \item[(a)] We have $\text{dist}(u,v_1) \le r_0$;
        \item[(b)] For all $1\le j\le n$ we have $2r_0 < \text{dist}(v_j,\{v_1,\dots,v_{j-1}\}) \le 10r_0$;
        \item[(c)] The length of the path satisfies $n\ge m/(10r_0)$.
        \item[(d)] For all $1\le j\le n$, we have either $v_i \in \cB$ or $|\Pi \cap \{v_i\}|>r_0$.
    \end{enumerate}
    Denote by $\cP_{n}(u)$ the set of all paths which satisfy $(a)$ and $(b)$ and have length $n$.
    By considering the different starting points in $B_{r_0}(u)$ for a path in $\cP_n(u)$, we have 
    \[
    |\cP_n(u)| \le b_{r_0} \cdot b_{10r_0}^{n-1} \le e^{Cr_0 n}
    \]
    for some $C>0$ which depends only on $G$.
    On the other hand, given any path $\{v_1,\ldots,v_n\}\in \cP_n(u)$ and $J \subset \{1,\dots,n\}$, we have that $|\{ v_j : j \in J\}^{+r_0}|=b_{r_0}|J|$
    and thus,~\eqref{eq:prob-bound-too-few-points} implies that 
    \begin{align*}
     \P(\{v_j:j\in J\} \subset \cB)
      &=\P\big(|\Pi'\cap B_{r_0}(v_j)| \le 0.9\, b_{r_0} \ \text{for all } j \in J\big) \\
      &\le \P\big(|\Pi'\cap \{v_j : j \in J\}^{+r_0}| \le 0.9\, b_{r_0}|J| \big) \le e^{-c b_{r_0} |J|} \, .
    \end{align*}
    Also, writing $J':=[n] \setminus J$, \eqref{eq:concentration_overcrowding} implies that
    \[ \P(|\Pi \cap \{v_j\}|>r_0\text{ for all }j \in J') \le \P(|\Pi \cap \{v_j : j \in J'\}|>r_0|J'|) \le e^{-\frac12|J'|r_0\log (r_0/2)}. \]
    Altogether, we get that
    \begin{align*}
        \P\big(\diam(\Gamma_{r_0}(u))\ge m\big)
         &\le \sum_{n\ge m/(10r_0)} |\cP_{n}(u)| \sum_{J \subset [n]} e^{-c b_{r_0} |J|} e^{-\frac12(n-|J|)r_0\log(r_0/2)} \\ 
         &\le \sum_{n\ge m/(10r_0)} e^{Cr_0 n} \cdot 2^n \cdot e^{-\min\{c b_{r_0},\frac12 r_0\log (r_0/2)\} \cdot n} \le 2^{-cm/r_0} \, ,
    \end{align*}
    where the last inequality holds for $r_0$ large enough. This shows that $\text{diam}(\Gamma_{r_0}(u))$ is finite almost surely, and the proof of the lemma is completed.
\end{proof}

\section{Sufficient conditions for~\eqref{eq:assumption_total_order}}
\label{sec:sufficient_conditions_for_total_order}

In this section we provide a sufficient condition for~\eqref{eq:assumption_total_order} to hold. Before doing so, let us elaborate on why the latter is assumed in our main result Theorem~\ref{thm:factor_matching_for_transitive_non-amenable_graphs}.
As mentioned in the introduction, factor matchings need not exist for general graphs. To demonstrate this, consider, for instance, the infinite ladder graph with added diagonals edges (see Figure~\ref{figure:ladder_with_diagonals}). 
Formally, this is a Cayley graph on $\Z\times\Z_2$ with generating set 
\[
\big\{(-1,0), (1,0), (0,1), (1,1), (-1,1)\big\}. 
\]  
The map which flips two vertical vertices and leaves all other vertices put is a graph automorphism. Therefore, if $|\Pi \cap \{v_0\}|=|\Pi \cap \{v_1\}|$ for some vertical pair of vertices $\{v_0,v_1\}$, then no equivariant matching between $\Pi$ and $V$ can have a matched edge crossing $\{v_0,v_1\}$ (i.e., every matched edge must be either contained in $\{v_0,v_1\}$ or disjoint from it).
On the other hand, if $|\Pi \cap \{v_0,v_1\}| \neq 2$ then, clearly, there must be a matched edge that crosses $\{v_0,v_1\}$.
Of course, for the Poisson process (or any reasonable perturbed vertex set), almost surely, both conditions will occur simultaneously, and hence an equivariant matching would be impossible in this case. 
An example of such a non-amenable Cayley graph is obtained in the same manner by replacing $\Z$ with a free group.
By requiring~\eqref{eq:assumption_total_order}, we ensure that the symmetry between different vertices is broken, so that such issues do not arise.
\vspace{3mm}
\begin{figure}[h!]
    \begin{center}	\scalebox{0.4}{\includegraphics{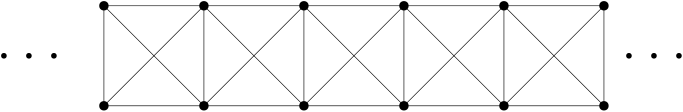}}
    \end{center}
    \caption{The infinite ladder graph with added diagonal edges.}
    \label{figure:ladder_with_diagonals}
\end{figure}

We now give a sufficient condition for the assumption~\eqref{eq:assumption_total_order}.
In fact, we will construct the required process $(\cO_v)$ as a factor of a single point process $\Pi$, rather than the pair $(\Pi,\Pi')$.
Recall that a perturbed vertex set is a multiset of the form $\Pi=\{X_v : v \in V\}$, where $(X_v)_{v \in V}$ are independent random variables taking values in $V$, and $\gamma(X_v)$ has the same law as $X_{\gamma(v)}$ whenever $\gamma$ is an automorphism of $G$. Denote by $S_r(v)$ the sphere of radius $r$ around a vertex $v$ in $G$.
        
\begin{lemma}\label{lem:distinct-point-counts}
    Let $G=(V,E)$ be a locally finite transitive connected graph such that $S_r(v) \neq S_r(u)$ for every distinct $u,v\in V$ and $r\ge 0$.    
    Let $\Pi$ be either a Poisson process, or a perturbed vertex set satisfying that there exists $c>0$ such that for every distinct $u,v\in V$ and $r\ge 0$ there exists $w \in V$ such that 
    \begin{equation} \label{eq:lem:distinct_point_count_condition_on_w}
        \text{\normalfont Var}\Big(\mathbf{1}_{\{X_{w} \in S_r(u) \setminus S_r(v)\}} \Big) \ge c.
    \end{equation}
    Then a.s.\ for any two distinct vertices $u,v\in V$ there exists $r\ge 0$ such that $|\Pi \cap S_r(v)| \neq |\Pi \cap S_r(u)|$.
\end{lemma}
\noindent
As we already mentioned in the introduction, it is straightforward to verify the conditions of Lemma~\ref{lem:distinct-point-counts} when $G$ is a $d$-regular tree and $\Pi$ is a Poisson process or a non-degenerate perturbed vertex set.

\begin{corollary}
    Under the assumptions of Lemma~\ref{lem:distinct-point-counts} we have that~\eqref{eq:assumption_total_order} holds.
\end{corollary}
\begin{proof}
    The map
    \[
    Z_v = \big(|\Pi\cap S_r(v)|\big)_{k=0}^\infty
    \]
    is a $\N^\N$-valued factor. By Lemma~\ref{lem:distinct-point-counts}, almost surely $Z_u\not=Z_v$ for all $u\not=v$. The corollary follows by choosing a measurable injection $\psi \colon \N^\N \to [0,1]$ and taking the composition $\cO_v = \psi(Z_v)$, which gives the desired factor $(\cO_v)$ of distinct real numbers. 
    Such a $\psi$ is given, for instance, by
    \[
    \psi(a_1,a_2,\ldots) = \sum_{k=1}^\infty \frac{d_k}{2^k} \, , 
    \]
    with the binary digits $(d_k)_{k=1}^\infty$ defined by
    \[
    d_1d_2d_3\ldots = \underbrace{1\cdots1}_{a_1\text{ times}}\,0\,
    \underbrace{1\cdots1}_{a_2\text{ times}}\,0\,
    \underbrace{1\cdots1}_{a_3\text{ times}}\,0\cdots. \qedhere
    \]
\end{proof}

\begin{proof}[Proof of Lemma~\ref{lem:distinct-point-counts}]
    We handle the case when $\Pi$ is a perturbed vertex set; the case when $\Pi$ is a Poisson process is similar (and much simpler).    
    Let $\{r_k\}$ be a subsequence that we choose later and denote by 
    \[
    E_k = \big\{ |\Pi\cap S_{r_k}(u)| \neq |\Pi\cap S_{r_k}(v)| \big\} \, .
    \]
    We want to show that $ \P(\bigcup_{k} E_k ) = 1$, which will follow immediately once we show that 
    \begin{equation} \label{eq:proof_of_lemma_distinct_points_lower_bound_on_conditional_prob}
        \inf_{k\ge 1} \P\big( E_k \mid E_1^c\cap \ldots \cap E_{k-1}^c \big) \ge \tfrac14 c \, .
    \end{equation}
    Fix some $k\ge 1$ and denote by $F_k = E_1^c\cap \ldots \cap E_{k-1}^c$. We can choose $r_k \ge 4 \max\{ r_{k-1}, \text{dist}(u,v)\}$ large enough so that 
    \begin{equation*}
        \P \big(\text{dist}(X_w,w) \ge \tfrac12 r_k \big) \le \tfrac12 \min\{c,\P(F_k)\} \, ,
    \end{equation*}
    for all $w$. By our assumption, there exists $w$ such that~\eqref{eq:lem:distinct_point_count_condition_on_w} holds with $r=r_k$.
    In particular, $\P(X_w \in S_{r_k}(u) \setminus S_{r_k}(v)) \ge c$, and hence, $\text{dist}(w, S_{r_k}(u) \setminus S_{r_k}(v)) \le r_k/2$. Thus, $\text{dist}(w,u) \ge r_k/2 \ge r_{k-1}$, and $\text{dist}(w,v) \ge \text{dist}(w,u) - \text{dist}(u,v) \ge r_k/4 \ge r_{k-1}$.
    Therefore,
    \begin{equation} \label{eq:proof_of_lemma_distinct_point_choice_of_r_k}
        \P \big(X_w \in B_{r_{k-1}}(u) \cup B_{r_{k-1}}(v) \big) \le \tfrac12 \min\{c,\P(F_k)\} \, ,
    \end{equation}
    
    Denoting by $\Pi^w =\{ X_v \, : \, v\not=w \}$, we have
    \begin{multline*}
    \mathbf{1}_{E_k} \ge \mathbf{1}_{\{X_w \in S_{r_k}(u) \setminus S_{r_k}(v)\}} \cdot \mathbf{1}_{\{|\Pi^w \cap S_{r_k}(v)| \le |\Pi^w \cap S_{r_k}(u)| \}} \\ +  \mathbf{1}_{\{X_w \notin S_{r_k}(u) \setminus S_{r_k}(v)\}} \cdot \mathbf{1}_{\{|\Pi^w \cap S_{r_k}(v)| > |\Pi^w \cap S_{r_k}(u)|\}} \, .
    \end{multline*}
    Noting that $X_w$ is independent of $\Pi^w$, we can take the conditional expectation over $F_k$ and observe from the above that
    \begin{equation} \label{eq:proof_of_lemmma_distinct_points_lower_bound_after_case_decomposition}
    \P(E_k \mid F_k )\ge \min\Big\{ \P\big(X_w \in S_{r_k}(u) \setminus S_{r_k}(v) \mid F_k \big) , \, \P\big(X_w \not \in S_{r_k}(u) \setminus S_{r_k}(v) \mid F_k \big)  \Big\} \, .    
    \end{equation}
    We now lower bound each of the terms on the right-hand side.
    The value of $X_w$ is irrelevant for the occurrence of $F_k$ on the event that $X_w \notin B_{r_{k-1}}(u) \cup B_{r_{k-1}}(v)$, and hence, for any $A \subset V \setminus (B_{r_{k-1}}(u) \cup B_{r_{k-1}}(v))$,
    \begin{align*}
        \P\big(X_w \in A \mid F_k \big) &= \P\big(X_w \in A \big) \cdot \frac{\P\big(F_k \mid X_w \in A \big)}{\P(F_k)} \\
        &= \P\big(X_w \in A \big) \cdot \frac{\P\big(F_k \mid X_w \not \in B_{r_{k-1}}(u) \cup B_{r_{k-1}}(v) \big)}{\P(F_k)} \\
        & \ge  \P\big(X_w \in A \big) \cdot \frac{\P(F_k) - \P\big( X_w \in B_{r_{k-1}}(u) \cup B_{r_{k-1}}(v) \big)}{\P(F_k)} \stackrel{\eqref{eq:proof_of_lemma_distinct_point_choice_of_r_k}}\ge \frac12 \, \P\big(X_w \in A \big) \, .
    \end{align*}
    On the other hand, taking $A=S_{r_k}(u) \setminus S_{r_k}(v)$, we get that $\P\big(X_w \in S_{r_k}(u) \setminus S_{r_k}(v) \mid F_k \big) \ge c/2$.
    Taking $A=(S_{r_k}(u) \setminus S_{r_k}(v))^c \setminus (B_{r_{k-1}}(u) \cup B_{r_{k-1}}(v))$, we get that
    \begin{align*}
       \P\big(X_w \notin S_{r_k}(u) \setminus S_{r_k}(v) \mid F_k \big) &\ge \P\big(X_w \in  (S_{r_k}(u) \setminus S_{r_k}(v))^c \setminus (B_{r_{k-1}}(u) \cup B_{r_{k-1}}(v))  \mid F_k \big) \\ & \ge \frac{1}{2} \P\big(X_w \in  (S_{r_k}(u) \setminus S_{r_k}(v))^c \setminus (B_{r_{k-1}}(u) \cup B_{r_{k-1}}(v))  \big) \\ &\ge \frac{1}{2} \Big(\P\big(X_w \in  (S_{r_k}(u) \setminus S_{r_k}(v))^c   \big) -  \P \big(X_w\in (B_{r_{k-1}}(u) \cup B_{r_{k-1}}(v) \big) \Big) \\  &\stackrel{\eqref{eq:proof_of_lemma_distinct_point_choice_of_r_k}}\ge \frac14 c . 
    \end{align*}
    Thus, \eqref{eq:proof_of_lemma_distinct_points_lower_bound_on_conditional_prob} follows from~\eqref{eq:proof_of_lemmma_distinct_points_lower_bound_after_case_decomposition}, and we are done. 
\end{proof}

\bibliographystyle{abbrv}
\bibliography{matching}

\bigskip 
\bigskip

\noindent 
Yinon Spinka \\ School of Mathematical Sciences, Tel Aviv University, Israel. \\ Email: {\tt yinonspi@tauex.tau.ac.il} 

\bigskip
\bigskip

\noindent 
Oren Yakir \\ Department of Mathematics, Massachusetts Institute of Technology, USA. \\ Email: {\tt oren.yakir@gmail.com}

\end{document}